\documentclass[11pt,reqno]{amsart}

\usepackage[T1]{fontenc}
\usepackage[utf8]{inputenc}
\usepackage{lmodern}
\usepackage{microtype}
\usepackage{booktabs}
\usepackage[margin=1.12in]{geometry}
\usepackage{setspace}
\usepackage{enumitem}
\usepackage{xcolor}
\usepackage[normalem]{ulem}
\usepackage{amsmath,amssymb,amsfonts,amsthm,mathtools}
\usepackage{mathrsfs}
\usepackage{bbm}
\usepackage{bm}
\numberwithin{equation}{section}
\allowdisplaybreaks

\usepackage{tikz}
\usetikzlibrary{arrows.meta,positioning,calc}

\usepackage[numbers,sort&compress]{natbib}
\usepackage[
  colorlinks=true,
  linkcolor=blue,
  citecolor=blue,
  urlcolor=blue
]{hyperref}
\usepackage[nameinlink,capitalize,noabbrev]{cleveref}

\newif\ifdraftversion
\draftversionfalse
\ifdraftversion
  \usepackage[pagewise]{lineno}
  \linenumbers
\fi

\theoremstyle{plain}
\newtheorem{theorem}{Theorem}[section]

\newtheorem{proposition}[theorem]{Proposition}
\newtheorem{corollary}[theorem]{Corollary}
\newtheorem{lemma}[theorem]{Lemma}

\theoremstyle{definition}

\theoremstyle{remark}
\newtheorem{remark}[theorem]{Remark}

\crefname{theorem}{Theorem}{Theorems}
\crefname{mainthm}{Main Theorem}{Main Theorems}
\crefname{proposition}{Proposition}{Propositions}
\crefname{corollary}{Corollary}{Corollaries}
\crefname{lemma}{Lemma}{Lemmas}
\crefname{definition}{Definition}{Definitions}
\crefname{assumption}{Assumption}{Assumptions}
\crefname{remark}{Remark}{Remarks}
\crefname{notation}{Notation}{Notations}
\crefname{equation}{equation}{equations}
\crefname{section}{Section}{Sections}
\crefname{appendix}{Appendix}{Appendices}

\newcommand{\spn}{\operatorname{sp}}
\newcommand{\dd}{\,\mathrm d}
\newcommand{\one}{\mathbf{1}}
\newcommand{\E}{\mathbb E}
\newcommand{\Prob}{\mathbb P}

\newcommand{\Mp}{M_{\mathrm p}}

\makeatletter
\newcommand{\proofstep}[1]{%
  \par
  \addvspace{\medskipamount}%
  \textit{#1\@addpunct{.}}\enspace\ignorespaces
}
\makeatother

\title[Critical Erd\H{o}s--R\'enyi Component Process]%
{Uniform High Order Factorial Moment Bounds for the
Critical Erd\H{o}s--R\'enyi Component Process}

\author{Wen Sun}
\address{School of Mathematical Sciences, University of Science and
Technology of China, Hefei 230026, China}
\email{wensun.ustc@gmail.com}
\date{}


\begin{document}
\begin{abstract}
We give a finite \(n\) enumerative derivation of the local surplus marked
point process form of Aldous's critical window limit for the
Erd\H{o}s--R\'enyi random graph~\cite{Aldous1997}.  For
\(p_n=n^{-1}+\lambda n^{-4/3}\), let \(\Xi_n\) place an atom at the rescaled
size and surplus of each component.  Exact component enumeration yields the
limiting factorial correlation densities and the uniform bound
\[
  \mathbb E[(\Xi_n(K))_q]\le C_K^q e^{-c_Kq^3},
  \qquad q\ge1,
\]
for every compact marked window \(K\), uniformly in admissible \(n\).  The
cubic order is optimal for the limiting factorial measures, and the estimate
persists after summing over all surpluses on compact size intervals.  It yields
overcrowding and local exponential moment bounds, quantitative truncation of
the finite \(n\) Laplace functional expansion, and local point process
convergence.  Using the Janson--Spencer Palm
description~\cite{JansonSpencer2007} and a classical all excess estimate, we
also recover ordered \(\ell^2\) component size convergence.
\end{abstract}

\maketitle

\section{Introduction}\label{sec:introduction}

The phase transition in the Erd\H{o}s--R\'enyi random graph \(G(n,p)\)
originates in the work of Erd\H{o}s and R\'enyi~\cite{ER1960}.  Inside the
critical window, the component size sequence and its surplus marks exhibit a
rich structure that has driven much of the enumerative and probabilistic
analysis of random graphs.  Early developments include
Bollob\'as~\cite{Bollobas1984},
\L uczak~\cite{Luczak1990}, Janson, Knuth,
\L uczak and Pittel~\cite{JKLP1993}, and \L uczak, Pittel and
Wierman~\cite{LPW1994}; see also general accounts in
\cite{Bollobas2001,JLR2000,vdH2017}.  The limiting component process is a
central object for the multiplicative coalescent and continuum critical random
graphs; see, for example, \cite{ABG2010,ABG2012}.

Fix \(\lambda\in\mathbb R\) and put
\begin{equation}\label{eq:pn-intro}
  p_n=n^{-1}+\lambda n^{-4/3}.
\end{equation}
Choose \(n_\lambda\) so that \(p_n\in(0,1)\) for every \(n\ge n_\lambda\).
Let \(\mathcal C(G)\) be the set of connected components of a graph \(G\).
For a component \(C\), write \(|C|\) for its number of vertices, \(e(C)\) for
its number of edges, and
\[
  \spn(C)=e(C)-|C|+1
\]
for its surplus.  This paper studies the local marked point process
\begin{equation}\label{eq:Xi-intro}
  \Xi_n
  =
  \sum_{C\in\mathcal C(G(n,p_n))}
  \delta_{(n^{-2/3}|C|,\spn(C))}
\end{equation}
on \(E=(0,\infty)\times\mathbb Z_{\ge0}\).  Thus each component contributes
an atom carrying its \(n^{-2/3}\)-rescaled size and surplus.  The aim is to
give a finite \(n\) enumerative derivation of the local surplus marked
point process limit and, more importantly, to obtain estimates uniform in the
factorial order.  We also show how all surplus enumeration and a classical
all excess estimate upgrade the local result to ordered \(\ell^2\)
component size convergence at each fixed \(\lambda\).

\paragraph{Classical convergence and fixed order factorial moments.}
Two classical routes underlie the convergence theory for critical random
graph components.  Aldous~\cite[Theorem~3 and Corollary~2]{Aldous1997}
used the exploration process:
after parabolic centering and diffusive scaling, the breadth first walk
converges to Brownian motion with parabolic drift, whose reflected excursions
give the limiting component sizes and whose excursion areas govern the
surplus marks.  A second route, due to \L uczak, Pittel and
Wierman~\cite{LPW1994,JLR2000}, is enumerative and studies critical
components through labelled graph enumeration and their internal structure.
This line gives limits for fixed numbers of the largest components and their
complexities, and hence local size point process information.  Aldous's
exploration process theorem gives the stronger convergence of the whole
ranked size sequence in \(\ell^2_\downarrow\).

Later formulations put these limits closer to the point process language used
here.  Janson and Spencer~\cite{JansonSpencer2007} described the limiting
point process through its intensity, independent complexity marks, Palm shift
and factorial moments.  Van der Hofstad, Kager and M\"uller~\cite{HKM2009}
proved a local limit theorem for the ranked largest components and stated its
surplus marked extension.  See also \cite{BJR2007,ABG2010,ABG2012} for
broader phase transition theory and continuum critical random graphs.

The present paper uses enumeration at the level of simultaneous component
tuples.  Instead of decomposing the internal structure of a critical
component, we count ordered collections of vertex disjoint components with
prescribed sizes and surpluses and interpret the exact finite \(n\) formula
as a factorial correlation calculation.  Related finite \(n\) fixed window
moment estimates appear in
Janson--Spencer~\cite[Proposition~1.4 and Lemma~5.1]{JansonSpencer2007}:
Proposition~1.4 treats the total rescaled mass and count above a fixed
cutoff, while Lemma~5.1 gives the corresponding fixed order moment bounds.
The new point is to keep the dependence on the factorial order \(q\) explicit
and to prove a single finite \(n\) estimate uniform in \(q\).  For the point
process \(\Xi_n\) on \(E=(0,\infty)\times\mathbb Z_{\ge0}\), this is the
estimate needed to control the complete
inclusion and exclusion expansion for the Laplace functional
\[
  \E e^{-\langle f,\Xi_n\rangle}
  =
  \sum_{q=0}^\infty\frac{(-1)^q}{q!}
  \int_{E^q}\prod_{i=1}^q\bigl(1-e^{-f(z_i)}\bigr)
  \,\dd\alpha_{\Xi_n}^{(q)}(z_1,\ldots,z_q),
\]
where \(\alpha_{\Xi_n}^{(q)}\) denotes the \(q\)th factorial moment measure
of \(\Xi_n\), to justify its termwise convergence
at finite \(n\), and to extract quantitative consequences such as
overcrowding bounds, generating functional truncation, and versions uniform
over compact intervals of \(\lambda\).  Combined with classical all excess
estimates, this local calculation also recovers the ordered \(\ell^2\) limit
for each fixed \(\lambda\).

\paragraph{Main result and consequences.}
Let \(K\Subset E\) be a compact marked window.  Thus its size coordinate lies
in a compact subinterval of \((0,\infty)\), and only finitely many surplus
values occur.  The main finite \(n\) estimate is the
factorial order uniform bound
\begin{equation}\label{eq:main-intro}
  \mathbb E[(\Xi_n(K))_q]\le C_K^q e^{-c_Kq^3}
\end{equation}
for suitable constants \(C_K,c_K>0\), depending only on \(K\) and
\(\lambda\), all \(n\ge n_\lambda\), and all \(q\ge1\); see
Theorem~\ref{thm:uniform-bound}.  The cubic power of \(q\) is optimal at the
level of the limiting factorial measures
(Proposition~\ref{prop:cubic-optimal}).  This estimate is the quantitative
ingredient behind the convergence proof and the consequences below.

\begin{enumerate}[label=(\arabic*),leftmargin=2.3em]
\item The complete inclusion and exclusion expansion for the Laplace functional
      is absolutely controlled at finite \(n\).  This gives tightness,
      uniqueness of subsequential limits, and local marked point process
      convergence before the Brownian limit is invoked
      (Theorem~\ref{thm:main}(ii));
\item Local overcrowding has a cubic exponential upper tail: for every
      compact marked window \(K\Subset E\),
      \[
        \Prob\{\Xi_n(K)\ge m\}
        \le \frac{C_K^m}{m!}\,e^{-c_Km^3},
        \qquad m\ge1,
      \]
      uniformly in \(n\).  The same estimate also gives uniform local
      exponential moments
      (Corollary~\ref{cor:finite-n-consequences});
\item Finite \(n\) generating functional expansions have exponentially
      accurate truncations: the tail beyond order \(Q\) is at most
      \(Ae^{-dQ^3}\), uniformly in \(n\) and in bounded test functions on a
      fixed compact support (Corollary~\ref{cor:finite-n-consequences});
\item After summing over all surplus values, the same cubic exponential
      high order bound holds on every compact size interval.  Together with
      the small component square mass estimate
      \(\sup_{n\ge n_\lambda}\E
      \sum_{i:X_{n,i}\le\varepsilon}X_{n,i}^2
      =O(\sqrt\varepsilon)\)
      and a classical all excess macroscopic mass bound, this recovers the
      ordered \(\ell^2\) component size limit for each fixed \(\lambda\)
      (Corollary~\ref{cor:l2-convergence} and
      Propositions~\ref{prop:unmarked-local}--\ref{prop:macro-mass}).
\end{enumerate}

The limit is identified with Aldous's parabolic drift Brownian excursion
process with Poisson area marks through the classical Janson--Spencer Palm
description \cite{JansonSpencer2007}; see
Corollary~\ref{cor:aldous-identification}.

\paragraph{Method.}
The proof starts from a finite graph calculation.  For each factorial order
\(q\), we enumerate ordered, vertex disjoint component tuples with prescribed
sizes and surplus labels, and we keep the dependence on \(q\) explicit.  The
argument has four main ingredients:
\begin{enumerate}[label=(\roman*),leftmargin=2.3em]
\item the exact finite \(n\) tuple formula, which gives the factorial moment
      formula for components with prescribed sizes and surpluses
      (Proposition~\ref{prop:exact});
\item Wright's fixed surplus enumeration
      \cite{Wright1977,Wright1980}, combined with the critical window
      expansion, which gives the limiting \(q\)-point correlation density
      (Proposition~\ref{prop:local});
\item the uniform high order bound \eqref{eq:main-intro}
      (Theorem~\ref{thm:uniform-bound}), which controls the
      inclusion and exclusion expansion for the Laplace functional and yields
      local marked point process convergence (Subsection~\ref{sec:laplace});
\item an all surplus connected graph bound and a classical all excess
      macroscopic mass estimate, which remove the surplus and size cutoffs
      and yield ordered \(\ell^2\) convergence
      (Section~\ref{sec:ordered}).
\end{enumerate}
At the identification stage, the classical Janson--Spencer Palm description
\cite{JansonSpencer2007} matches the limiting factorial measures with the
surplus marked Brownian excursion process of Aldous~\cite{Aldous1997}; see
Corollary~\ref{cor:aldous-identification}.

The main cancellation is already visible in the exact tuple formula.  Suppose
the selected components have total size \(K=k_1+\cdots+k_q\).  In the
logarithmic asymptotics, the linear and quadratic contributions coming from
the connected graph counts, the falling factorial, the present edge factor
and the absent edge factor cancel, leaving
\[
  -\frac{K^3}{6n^2}
  +\frac{\lambda K^2}{2n^{4/3}}
  -\frac{\lambda^2K}{2n^{2/3}}.
\]
After the scaling \(K n^{-2/3}\to x_1+\cdots+x_q\), this term becomes
\(-F(x_1+\cdots+x_q,\lambda)\).  This explains the factor that is not a product
\(\exp\{-F(x_1+\cdots+x_q,\lambda)\}\) in the limiting factorial densities.
It is the simultaneous counterpart of the successive parameter shifts used in
the Janson--Spencer and van der Hofstad--Kager--M\"uller descriptions
\cite{JansonSpencer2007,HKM2009}.  The reusable part of the argument is the
conversion of this simultaneous \(q\)-component formula into an
summable Laplace functional expansion that is uniform in the order.

The local topology is the usual vague topology on locally finite point
measures \cite[Chapter~4]{Kallenberg2017}.  Accordingly, compact subsets of
\(E\) stay away from zero in the size coordinate, are bounded above, and
contain only finitely many surplus values.  Section~\ref{sec:ordered} adds
the small size, large size and surplus tail estimates needed to pass from
local marked convergence to the ordered \(\ell^2\) component size limit for a
fixed value of \(\lambda\).

The connected graph ingredient is Wright's fixed excess enumeration
\cite{Wright1977,Wright1980}; its relation to Brownian excursion area was
made explicit by Spencer~\cite{Spencer1997} and surveyed by
Janson~\cite{Janson2007}.  The ordered \(\ell^2\) extension additionally uses
the all surplus bound of Janson--Spencer~\cite[(3.7)]{JansonSpencer2007} and
their all excess estimate for the total mass of complex components
\cite[(5.3)]{JansonSpencer2007}.  Related component size tail estimates,
including cubic critical window exponents for unusually large components,
appear in
\cite{NachmiasPeres2010,Roberts2018,DeAmbroggioRoberts2022}; see also
\cite{DeAmbroggio2022,DeAmbroggio2024} for elementary approaches to critical
component sizes.  Those results concern largest component or single component
bounds and related component size estimates, whereas the present estimates
control local factorial correlations uniformly in all orders.

\paragraph{Organization.}
Section~\ref{sec:setup} recalls the notation, introduces the candidate
factorial moment densities and states the main theorem.
Section~\ref{sec:finite-n} proves the marked local limit from the finite
\(n\) tuple formula, the local asymptotic, the uniform high order estimate,
and the Laplace functional argument; it also identifies the limit with
Aldous's process.  Section~\ref{sec:ordered} removes the local truncations
and proves ordered \(\ell^2\) convergence.

\section{Setup and main result}\label{sec:setup}

We keep the notation introduced in Section~\ref{sec:introduction}: \(p_n\) is
given by \eqref{eq:pn-intro}, \(n\ge n_\lambda\), and \(\Xi_n\) is the marked
point measure \eqref{eq:Xi-intro} on
\(E=(0,\infty)\times\mathbb Z_{\ge0}\).  Compact test windows will always
select components whose rescaled sizes stay in a bounded interval away from
zero and whose surplus marks lie in a finite set.
When the dependence on the parameter is relevant, we write
\(p_{n,\lambda}\) and \(\Xi_{n,\lambda}\); otherwise \(\lambda\) is fixed and
suppressed as above.

\subsection{Factorial moment measures}

Let \(\mathcal M(E)\) be the space of locally finite Radon measures on \(E\),
equipped with the vague topology.  This is a Polish space by the general
theory of random measures; see Kallenberg~\cite[Chapter~4]{Kallenberg2017}.
Let \(\Mp(E)\subset\mathcal M(E)\) denote the locally finite integer valued
Radon measures.  This is the standard state space for locally finite point
processes; see also Daley--Vere-Jones~\cite[Section~11.1]{DVJ2008}.  The
subspace \(\Mp(E)\) is closed, hence Polish.  Indeed, if
\(\mu_j\in\Mp(E)\) and \(\mu_j\to\mu\) vaguely, then
\(\mu_j(B)\to\mu(B)\) for every relatively compact Borel set \(B\) with
\(\mu(\partial B)=0\).  Hence \(\mu(B)\in\mathbb Z_{\ge0}\) for such \(B\).
Choose a countable generating ring \(\mathcal R\) of relatively compact
\(\mu\)-continuity sets.  For any relatively compact Borel set \(A\), regularity
gives \(R_m\in\mathcal R\) with \(\mu(A\triangle R_m)\to0\).  Since
\(\mu(R_m)\in\mathbb Z_{\ge0}\) and \(\mu(R_m)\to\mu(A)<\infty\), we get
\(\mu(A)\in\mathbb Z_{\ge0}\).  Thus \(\mu\in\Mp(E)\).

For a locally finite point measure \(\xi\), sums
\(\sum_{z_1,\ldots,z_q\in\xi}^{\neq}\) are over ordered tuples of distinct
\emph{atom indices}.  Thus atoms at the
same spatial location are still distinguished if they have different
multiplicity indices.  The \(q\)th factorial moment measure of a point
process \(\Xi\) is defined by
\begin{equation}\label{eq:fmm-definition}
  \int_{E^q}H\,\dd\alpha_\Xi^{(q)}
  =
  \E\sum_{z_1,\ldots,z_q\in\Xi}^{\neq}
  H(z_1,\ldots,z_q)
\end{equation}
for nonnegative measurable \(H\).  We use
\((m)_q=m(m-1)\cdots(m-q+1)\).

\subsection{Wright constants and the candidate densities}

Let \(c(k,m)\) be the number of connected labelled simple graphs on \(k\)
vertices with \(m\) edges.  For every fixed \(r\ge0\), Wright proved
\begin{equation}\label{eq:wright}
  c(k,k+r-1)
  \sim
  w_r k^{k+3r/2-2},
  \qquad k\to\infty,
\end{equation}
where \(w_r>0\); see Wright~\cite{Wright1977}, and, in exactly this
normalization, Janson--Spencer~\cite[(3.1)]{JansonSpencer2007}, with related
later asymptotic results in~\cite{Wright1980}.  In particular, \(w_0=1\), consistently with
Cayley's formula \(c(k,k-1)=k^{k-2}\).
We use the notation \(\Psi(t)=\sum_{r=0}^\infty w_rt^r\).

Let
\begin{equation}\label{eq:F-def}
  F(x,\lambda)
  =
  \frac{x^3}{6}-\frac{\lambda x^2}{2}+\frac{\lambda^2x}{2}
  =
  \frac12\int_0^x(u-\lambda)^2\,\dd u
\end{equation}
and put
\begin{equation}\label{eq:rho-def}
  \rho_r(x)=\frac{w_r}{\sqrt{2\pi}}x^{3r/2-5/2},
  \qquad x>0,\quad r\ge0.
\end{equation}
For \(q\ge1\), define
\begin{equation}\label{eq:mq-def}
  m_\lambda^{(q)}\bigl((x_1,r_1),\ldots,(x_q,r_q)\bigr)=\exp\!\left\{-F\!\left(\sum_{i=1}^q x_i,\lambda\right)\right\}\prod_{i=1}^q\rho_{r_i}(x_i).
\end{equation}
Let \(\alpha_\lambda^{(q)}\) be the measure on \(E^q\) with this density
with respect to Lebesgue measure in each size coordinate and counting
measure in each surplus coordinate.  These measures are locally finite,
since every compact subset of \(E^q\) stays away from zero in all size
coordinates and contains only finitely many surplus vectors.

\subsection{The Aldous marked excursion process}

We recall the limiting object in Aldous's critical window theorem
\cite[Corollary~2]{Aldous1997}.  Let \(B\) be standard
Brownian motion and set
\begin{equation}\label{eq:parabolic-drift}
  W_\lambda(t)=B(t)+\lambda t-\frac{t^2}{2},
  \qquad
  \overline W_\lambda(t)
  =W_\lambda(t)-\inf_{0\le s\le t}W_\lambda(s).
\end{equation}
Write \((\gamma_j)\) for the excursions of \(\overline W_\lambda\), let
\(\zeta_j\) be the lifetime of \(\gamma_j\), and let
\[
  A_j=\int_0^{\zeta_j}\gamma_j(s)\,\dd s.
\]
Conditional on the excursion paths, let \((N_j)\) be independent with
\(N_j\sim\operatorname{Poisson}(A_j)\), and define
\begin{equation}\label{eq:Aldous-process}
  \Xi_\lambda^{\mathrm A}
  =\sum_j\delta_{(\zeta_j,N_j)}.
\end{equation}

Let
\[
  \ell^2_\downarrow
  =\left\{\mathbf x=(x_1,x_2,\ldots):
    x_1\ge x_2\ge\cdots\ge0,\quad \sum_{i\ge1}x_i^2<\infty\right\}
\]
with the \(\ell^2\) metric.  Arrange the components of \(G(n,p_n)\) in
nonincreasing order of size, breaking finite \(n\) ties arbitrarily, append
zeros, and set
\begin{equation}\label{eq:Xn-def}
  \mathbf X_n
  =\bigl(n^{-2/3}|C_{n,1}|,n^{-2/3}|C_{n,2}|,\ldots\bigr).
\end{equation}
We write \(X_{n,i}=n^{-2/3}|C_{n,i}|\) for coordinate \(i\) of
\(\mathbf X_n\).
Similarly, let \(\boldsymbol\zeta_\lambda\) be the nonincreasing
rearrangement of the excursion lifetimes \((\zeta_j)\), padded by zeros when
viewed as a sequence.  Aldous's theorem gives
\(\boldsymbol\zeta_\lambda\in\ell^2_\downarrow\) almost surely and identifies
it as the classical \(\ell^2_\downarrow\) limit of \(\mathbf X_n\); the
Poisson variables above are the corresponding limiting surplus marks.  Thus
\(\Xi_\lambda^{\mathrm A}\) is the associated local marked point process.

\subsection{Main results}

We now state the main assertions.  The first theorem is the finite \(n\),
all order estimate.  The second gives the resulting local marked
point process convergence by Laplace functionals.  The two corollaries
identify the limit with Aldous's excursion process and recover the ordered
\(\ell^2\) component size limit.

\begin{theorem}[Uniform high order factorial bound]
\label{thm:uniform-bound}
Fix \(\lambda\in\mathbb R\), let \(p_n\) be given by
\eqref{eq:pn-intro}, and let \(\Xi_n\) be given by \eqref{eq:Xi-intro}.
For every compact \(K\Subset E\), there are constants
\(C_K,c_K>0\), depending only on \(K\) and \(\lambda\), such that
\begin{equation}\label{eq:main-uniform-bound}
  \mathbb E[(\Xi_n(K))_q]
  \le C_K^q e^{-c_Kq^3},
  \qquad q\ge1,
\end{equation}
uniformly in every \(n\ge n_\lambda\).
Moreover, if \(|\lambda|\le L<\infty\), then \(n_L\) can be chosen so that
for every compact \(K\Subset E\) there are \(C_K,c_K>0\), depending only on
\(K\) and \(L\), for which \eqref{eq:main-uniform-bound} holds uniformly in
all \(n\ge n_L\) and all \(|\lambda|\le L\), with \(p_n\) and \(\Xi_n\)
carrying the parameter \(\lambda\).
\end{theorem}

\begin{remark}
The cubic order in \eqref{eq:main-uniform-bound} is optimal at the level of
the limiting factorial measures; see Proposition~\ref{prop:cubic-optimal}
below.
\end{remark}

\begin{theorem}[Finite \(n\) enumerative convergence]\label{thm:main}
Fix \(\lambda\in\mathbb R\), let \(p_n\) be given by
\eqref{eq:pn-intro}, and let \(\Xi_n\) be given by \eqref{eq:Xi-intro}.

\begin{enumerate}[label=(\roman*),leftmargin=2.2em]
\item For every fixed \(q\ge1\), the factorial moment measures satisfy
\begin{equation}\label{eq:fmm-vague}
  \alpha_{\Xi_n}^{(q)}\xrightarrow{v}\alpha_\lambda^{(q)}
  \qquad\text{on }E^q.
\end{equation}
In fact, the stronger uniform local asymptotic in
Proposition~\ref{prop:local} holds.

\item There is a unique point process
\(\Xi_\lambda^{\mathrm{enum}}\) on \(E\) whose Laplace functional is
\begin{equation}\label{eq:enum-Laplace-theorem}
\begin{aligned}
  \mathbb E e^{-\langle f,\Xi_\lambda^{\mathrm{enum}}\rangle}
  =\sum_{q=0}^\infty\frac{(-1)^q}{q!}
  \int_{E^q}\prod_{i=1}^q(1-e^{-f(z_i)})
  \,\mathrm d\alpha_\lambda^{(q)}(\mathbf z)
\end{aligned}
\end{equation}
for every \(f\in C_c^+(E)\), and
\begin{equation}\label{eq:enum-convergence}
  \Xi_n\Rightarrow\Xi_\lambda^{\mathrm{enum}}
  \qquad\text{in }M_{\mathrm p}(E)
\end{equation}
under the vague topology.
\end{enumerate}
In \eqref{eq:enum-Laplace-theorem}, the \(q=0\) term is one: by convention
\(\alpha_\lambda^{(0)}\) is unit mass on the singleton \(E^0\), and the empty
product equals one.
\end{theorem}

The next two corollaries show that the finite \(n\) enumeration, together
with the classical Janson--Spencer identification and tail estimates,
recovers the fixed \(\lambda\) marked point process and ordered \(\ell^2\)
component size conclusions of
Aldous~\cite[Theorem~3 and Corollary~2]{Aldous1997}.

\begin{corollary}[Classical identification with the Aldous process]
\label{cor:aldous-identification}
The process constructed in Theorem~\ref{thm:main} satisfies
\[
  \Xi_\lambda^{\mathrm{enum}}
  \stackrel{d}{=}\Xi_\lambda^{\mathrm A}.
\]
Consequently,
\begin{equation}\label{eq:main-convergence}
  \Xi_n\Rightarrow\Xi_\lambda^{\mathrm A}
  \qquad\text{in }M_{\mathrm p}(E).
\end{equation}
\end{corollary}

\begin{corollary}[Ordered component sizes and largest marked components]
\label{cor:l2-convergence}
The following two conclusions hold.
\begin{enumerate}[label=(\roman*),leftmargin=2.3em]
\item \(\mathbf X_n\) satisfies
\begin{equation}\label{eq:l2-convergence}
  \mathbf X_n\Rightarrow\boldsymbol\zeta_\lambda
  \qquad\text{in }\ell^2_\downarrow.
\end{equation}
\item Arrange the atoms of \(\Xi_\lambda^{\mathrm A}\) in decreasing order of
their size coordinates and denote them by
\((\zeta_{\lambda,i},N_{\lambda,i})_{i\ge1}\).  For every fixed \(m\ge1\),
\begin{equation}\label{eq:ranked-marked-convergence}
  \bigl(n^{-2/3}|C_{n,i}|,\spn(C_{n,i})\bigr)_{i=1}^m
  \Rightarrow
  \bigl(\zeta_{\lambda,i},N_{\lambda,i}\bigr)_{i=1}^m.
\end{equation}
The conclusion is independent of the rule used to break equal size ties in
the finite graph.
\end{enumerate}
\end{corollary}

\section{Finite \texorpdfstring{\(n\)}{n} enumeration and the marked local
limit}
\label{sec:finite-n}

This section proves the marked local limit from finite graph enumeration.
The exact ordered tuple formula and its critical window asymptotic give
Theorem~\ref{thm:main}(i).  The factorial order uniform bound in
Theorem~\ref{thm:uniform-bound} then supplies the absolute control needed for
the Laplace functional proof of Theorem~\ref{thm:main}(ii).  The final
subsection identifies the resulting limit with Aldous's marked excursion
process.

\begin{lemma}[Uniform fixed surplus Wright bound]\label{lem:wright-bound}
For every integer \(R_0\ge0\), there is \(C=C(R_0)<\infty\) such that
\begin{equation}\label{eq:wright-upper}
  c(k,k+r-1)
  \le C k^{k+3r/2-2}
\end{equation}
for all \(k\ge1\) and \(0\le r\le R_0\), with the convention that the left
side is zero when the edge count is infeasible.
\end{lemma}

\begin{proof}
For each fixed \(r\), Wright's asymptotic \eqref{eq:wright} implies that
\[
  A_r:=\sup_{k\ge1}
  \frac{c(k,k+r-1)}{k^{k+3r/2-2}}<\infty,
\]
where infeasible edge counts contribute zero.  Since
\(\{0,\ldots,R_0\}\) is finite, \(C=\max_{0\le r\le R_0}A_r\) gives
\eqref{eq:wright-upper} uniformly for all \(k\ge1\) and
\(0\le r\le R_0\).
\end{proof}

\subsection{Exact tuple enumeration and local asymptotics}

This subsection derives the fixed order factorial densities from exact
component tuple enumeration and then proves Theorem~\ref{thm:main}(i).

Fix \(q\ge1\), integers \(k_i\ge1\), and surpluses \(r_i\ge0\).  Write
\(\mathbf k=(k_1,\ldots,k_q)\), \(\mathbf r=(r_1,\ldots,r_q)\), and
\begin{equation}\label{eq:KR-def}
  K=\sum_{i=1}^qk_i,\qquad R=\sum_{i=1}^qr_i.
\end{equation}
We use the notation
\begin{equation}\label{eq:Ln-def}
  L_n(\mathbf k,\mathbf r)=\E\sum_{C_1,\ldots,C_q\in\mathcal C(G(n,p_n))}^{\neq}\prod_{i=1}^q\one_{\{|C_i|=k_i,\ \spn(C_i)=r_i\}}.
\end{equation}

\begin{proposition}[Exact ordered tuple formula]\label{prop:exact}
If \(K\le n\) and
\(k_i+r_i-1\le\binom{k_i}{2}\) for every \(i\), then
\begin{equation}\label{eq:exact}
  L_n(\mathbf k,\mathbf r)=\frac{(n)_K}{\prod_{i=1}^q k_i!}\prod_{i=1}^q c(k_i,k_i+r_i-1)\,p_n^{K+R-q}(1-p_n)^{A},
\end{equation}
where
\begin{equation}\label{eq:A-def}
\begin{aligned}
  A
  &=K(n-K)
    +\sum_{1\le i<j\le q}k_ik_j
    +\sum_{i=1}^q
      \left\{\binom{k_i}{2}-(k_i+r_i-1)\right\}
  \\
  &=Kn-\frac{K^2}{2}-\frac{3K}{2}-R+q.
\end{aligned}
\end{equation}
If the feasibility conditions fail, then \(L_n(\mathbf k,\mathbf r)=0\).
\end{proposition}

\begin{proof}
Choose ordered pairwise disjoint vertex sets of cardinalities
\(k_1,\ldots,k_q\).  Their number is
\((n)_K/\prod_i k_i!\).  On set number \(i\), choose a connected labelled
graph with \(k_i+r_i-1\) edges.  Altogether this prescribes
\[
  \sum_{i=1}^q(k_i+r_i-1)=K+R-q
\]
present edges.

For the chosen subgraphs to be connected components, three classes of edges
must be absent: edges from the selected vertices to the remaining
\(n-K\) vertices; edges between two different selected sets; and all
unprescribed internal edges of the selected sets.  This gives the first
line of \eqref{eq:A-def}.  Since
\[
  \sum_{i<j}k_ik_j=\frac12\left(K^2-\sum_i k_i^2\right),
  \qquad
  \sum_i\binom{k_i}{2}=\frac12\left(\sum_i k_i^2-K\right),
\]
the first line simplifies to the second.  Independence of edges now gives
\eqref{eq:exact}.
\end{proof}

\begin{remark}
For \(q=1\), \eqref{eq:exact} is the exact single component expectation; its
critical window asymptotic is the single component calculation in
Janson--Spencer~\cite[(4.1)]{JansonSpencer2007}.  The proposition is the
simultaneous ordered \(q\) version of the same deletion calculation.
\end{remark}

\begin{proposition}[Local asymptotics]\label{prop:local}
Fix \(q\ge1\), a surplus vector
\(\mathbf r=(r_1,\ldots,r_q)\in\mathbb Z_{\ge0}^q\), and
\(0<a<b<\infty\).  Let \(k_i=k_i(n)\) range over integers such that
\begin{equation}\label{eq:local-grid}
  x_i:=k_i n^{-2/3}\in[a,b].
\end{equation}
Then, uniformly over all \(\mathbf k=(k_1,\ldots,k_q)\) satisfying
\eqref{eq:local-grid},
\begin{equation}\label{eq:local-asymptotic}
  L_n(\mathbf k,\mathbf r)=n^{-2q/3}m_\lambda^{(q)}\bigl((x_1,r_1),\ldots,(x_q,r_q)\bigr)(1+o(1)).
\end{equation}
The estimate is also uniform when \(\lambda\) ranges over a fixed compact
interval, with \(L_n\) formed using \(p_{n,\lambda}\).
\end{proposition}

\begin{proof}
Because \(q\) and \(\mathbf r\) are fixed and \(k_i\asymp n^{2/3}\), all
feasibility conditions in Proposition~\ref{prop:exact} hold for large
\(n\).  We expand its factors uniformly over \eqref{eq:local-grid}.

Wright's asymptotic \eqref{eq:wright} and Stirling's formula give
\begin{equation}\label{eq:connected-over-factorial}
  \frac{c(k_i,k_i+r_i-1)}{k_i!}=\frac{w_{r_i}}{\sqrt{2\pi}}e^{k_i}k_i^{3r_i/2-5/2}(1+o(1)),
\end{equation}
uniformly in \(i\) and in the grid points under consideration.  Uniformity
follows because \(k_i\ge an^{2/3}\to\infty\) and only finitely many
surpluses occur.

Recall \(K=\sum_i k_i\) and \(R=\sum_i r_i\).  Since \(K=O(n^{2/3})\), Taylor
expansion of \(\sum_{j=0}^{K-1}\log(1-j/n)\) yields
\begin{equation}\label{eq:falling-local}
  \log\frac{(n)_K}{n^K}=-\frac{K(K-1)}{2n}-\frac{K(K-1)(2K-1)}{12n^2}+O\!\left(\frac{K^4}{n^3}\right)=-\frac{K^2}{2n}-\frac{K^3}{6n^2}+O(n^{-1/3}).
\end{equation}
The error is uniform because \(K\le qbn^{2/3}\).

Writing \(p_n=n^{-1}(1+\lambda n^{-1/3})\), and recalling that
\(R-q\) is fixed, we have
\begin{equation}\label{eq:present-local}
  (K+R-q)\log p_n=-(K+R-q)\log n+\lambda K n^{-1/3}-\frac{\lambda^2K}{2n^{2/3}}+O(n^{-1/3}).
\end{equation}
Furthermore \(A=O(n^{5/3})\), and
\(\log(1-p_n)=-p_n-p_n^2/2+O(p_n^3)\).  Substituting
\eqref{eq:A-def} gives
\begin{equation}\label{eq:absent-local}
  A\log(1-p_n)=-K+\frac{K^2}{2n}-\lambda K n^{-1/3}+\frac{\lambda K^2}{2n^{4/3}}+O(n^{-1/3}).
\end{equation}
The contribution of \(-Ap_n^2/2\) is \(O(n^{-1/3})\), and the higher order
terms contribute \(O(n^{-4/3})\); hence the displayed error is uniform.

In the logarithm of \eqref{eq:exact}, the factor \(e^K\) from
\eqref{eq:connected-over-factorial} cancels the \(-K\) in
\eqref{eq:absent-local}; the term \(-K^2/(2n)\) in
\eqref{eq:falling-local} cancels the \(+K^2/(2n)\) in
\eqref{eq:absent-local}; and the term \(+\lambda Kn^{-1/3}\) in
\eqref{eq:present-local} cancels the \(-\lambda Kn^{-1/3}\) in
\eqref{eq:absent-local}.  The remaining exponential term is
\begin{equation}\label{eq:critical-exponent-local}
  -\frac{K^3}{6n^2}+\frac{\lambda K^2}{2n^{4/3}}-\frac{\lambda^2K}{2n^{2/3}}+O(n^{-1/3}).
\end{equation}
Since \(K n^{-2/3}=S:=\sum_i x_i\), this is
\(-F(S,\lambda)+o(1)\), uniformly.

Finally, the remaining powers of \(n\) are
\begin{equation}\label{eq:power-local}
  n^{q-R}\prod_{i=1}^q k_i^{3r_i/2-5/2}=n^{q-R}n^{R-5q/3}\prod_{i=1}^q x_i^{3r_i/2-5/2}=n^{-2q/3}\prod_{i=1}^q x_i^{3r_i/2-5/2}.
\end{equation}
Combining \eqref{eq:connected-over-factorial}--\eqref{eq:power-local}
proves \eqref{eq:local-asymptotic}.  When \(\lambda\) ranges over a compact
interval, every displayed Taylor remainder is uniform in \(\lambda\), which
proves the final assertion.
\end{proof}

\begin{proof}[Proof of Theorem~\ref{thm:main}(i)]
Let \(H\in C_c(E^q)\).  The defining identity
\eqref{eq:fmm-definition}, initially stated for nonnegative functions,
extends to such signed functions by positive and negative parts.  The support
of \(H\) is contained in
\[
  \bigl([a,b]\times\{0,\ldots,R_0\}\bigr)^q
\]
for suitable \(0<a<b<\infty\) and \(R_0<\infty\).  There are only finitely
many surplus vectors, and
\[
  \int_{E^q}H\,\dd\alpha_{\Xi_n}^{(q)}
  =\sum_{\mathbf r}\sum_{\mathbf k}
  H\bigl((k_i n^{-2/3},r_i)_{i=1}^q\bigr)
  L_n(\mathbf k,\mathbf r),
\]
where only the vectors in the support of \(H\) contribute.  For each
\(\mathbf r\), Proposition~\ref{prop:local} turns this expression into a
Riemann sum whose cell volume is \(n^{-2q/3}\).  The uniform asymptotic
\eqref{eq:local-asymptotic} and the continuity of \(H\) now give
\eqref{eq:fmm-vague}.
\end{proof}

\subsection{Uniform bound in the factorial moment order}

Fixed order asymptotics alone do not control the full finite \(n\)
inclusion--exclusion expansion for the Laplace functional.  We now prove the
all order estimate that yields absolute summability and quantitative
truncation bounds.

\begin{proof}[Proof of Theorem~\ref{thm:uniform-bound}]
It is enough to prove the estimate for rectangles of the form
\begin{equation}\label{eq:K0-def}
  K_0=[a,b]\times\{0,\ldots,R_0\}\Subset E,
  \qquad 0<a<b<\infty.
\end{equation}
Indeed, every compact \(K\Subset E\) is contained in such a rectangle, and
falling factorials are monotone under inclusion of counting variables.  We
therefore prove
\begin{equation}\label{eq:uniform-bound}
  \E[(\Xi_n(K_0))_q]\le C^qe^{-cq^3},\qquad q\ge1,
\end{equation}
with \(C,c\) depending only on \(a,b,R_0,\lambda\), uniformly in
\(n\ge n_\lambda\).

A contributing tuple satisfies
\begin{equation}\label{eq:tuple-range}
  an^{2/3}\le k_i\le bn^{2/3},
  \qquad 0\le r_i\le R_0.
\end{equation}
By the definition of \(L_n(\mathbf k,\mathbf r)\), the factorial moment is
the sum
\begin{equation}\label{eq:factorial-sum-K0}
  \E[(\Xi_n(K_0))_q]=\sum_{r_1,\ldots,r_q=0}^{R_0}\sum_{\substack{an^{2/3}\le k_i\le bn^{2/3}\\1\le i\le q}}L_n(\mathbf k,\mathbf r),
\end{equation}
where terms with \(K>n\) or infeasible edge counts are zero.  For a term in
this sum put
\[
  x_i=k_i n^{-2/3},\qquad K=\sum_i k_i,\qquad
  R=\sum_i r_i,\qquad S=Kn^{-2/3}=\sum_i x_i.
\]
Then \(aq\le S\le bq\).  Tuples with \(K>n\) contribute zero, so assume
\(K\le n\).  In particular,
\begin{equation}\label{eq:q-feasible}
  q\le a^{-1}n^{1/3}.
\end{equation}
Take \(n\) large enough that \(p_n\in(0,1/2)\) and
\(|\lambda|n^{-1/3}\le1/2\).

By Lemma~\ref{lem:wright-bound} and a lower Stirling bound,
\begin{equation}\label{eq:enumerative-upper}
  \frac{c(k_i,k_i+r_i-1)}{k_i!}
  \le C_1 e^{k_i}k_i^{3r_i/2-5/2}
\end{equation}
uniformly over \eqref{eq:tuple-range}.  Substitute this in
\eqref{eq:exact}, write
\(\eta=\lambda n^{-1/3}\), and extract all powers of \(n\).  For each fixed
admissible choice of \(\mathbf k\) and \(\mathbf r\), the corresponding
contribution \(L_n(\mathbf k,\mathbf r)\) to \(\E[(\Xi_n(K_0))_q]\) is at most
\begin{equation}\label{eq:summand-upper}
  C_2^q n^{-2q/3}
  \prod_{i=1}^q x_i^{3r_i/2-5/2}
  \exp\{H_n(K,R,q)\},
\end{equation}
where
\begin{equation}\label{eq:H-def}
  H_n(K,R,q)=K+\log\frac{(n)_K}{n^K}+(K+R-q)\log(1+\eta)+A\log(1-p_n),
\end{equation}
and \(A\) is given by \eqref{eq:A-def}.  We prove
\begin{equation}\label{eq:H-target}
  H_n(K,R,q)\le-c_1S^3+C_3q.
\end{equation}

\smallskip
\noindent\emph{Case 1: \(K\le n/2\).}
For \(0\le j/n\le1/2\),
\(\log(1-j/n)\le-j/n-j^2/(2n^2)\).  Summing gives
\begin{equation}\label{eq:falling-upper}
  \log\frac{(n)_K}{n^K}\le-\frac{K(K-1)}{2n}-\frac{K(K-1)(2K-1)}{12n^2}\le-\frac{K^2}{2n}-\frac{K^3}{6n^2}+C\left(\frac Kn+\frac{K^2}{n^2}\right).
\end{equation}
For fixed \(\lambda\),
\begin{equation}\label{eq:log-plus-bound}
  \log(1+\eta)
  =\eta-\frac{\eta^2}{2}+\varepsilon_n,
  \qquad |\varepsilon_n|\le C_\lambda n^{-1}.
\end{equation}
For a contributing tuple, \(A\ge0\), so
\(\log(1-p_n)\le-p_n=-(1+\eta)/n\) may be multiplied by \(A\).  Let
\(d=R-q\), so \(|d|\le(R_0+1)q\), and note that
\[
  A=Kn-\frac{K^2}{2}-\frac{3K}{2}-d.
\]
Keeping the terms paired makes the uniformity in \(q\) explicit.  From
\(A=Kn-K^2/2-3K/2-d\),
\begin{equation}\label{eq:absent-upper-expanded}
  K-\frac{1+\eta}{n}A=\frac{K^2}{2n}+\frac{3K}{2n}+\frac dn-\eta K+\frac{\eta K^2}{2n}+\frac{3\eta K}{2n}+\frac{\eta d}{n},
\end{equation}
whereas
\begin{equation}\label{eq:present-upper-expanded}
\begin{aligned}
  (K+d)\log(1+\eta)
  &=K\eta+d\eta-\frac{K\eta^2}{2}-\frac{d\eta^2}{2}
    +(K+d)\varepsilon_n.
\end{aligned}
\end{equation}
Thus \(K\eta\) cancels \(-\eta K\), and the term \(K^2/(2n)\) in
\eqref{eq:absent-upper-expanded} cancels the corresponding negative term in
\eqref{eq:falling-upper}.  To make the uniformity explicit, all unretained
terms are bounded by
\begin{equation}\label{eq:remainder-small-explicit}
 C_{\lambda,R_0}\left(
   \frac Kn+\frac{K^2}{n^2}+|d|n^{-1/3}+1\right).
\end{equation}
Indeed, the first two terms cover the remainder in
\eqref{eq:falling-upper}; the terms containing \(d\) use
\(|d|\le(R_0+1)q\); and \((K+d)\varepsilon_n\),
\(\eta K/n\), \(\eta d/n\), and their smaller companions are covered using
\(K\le n/2\), \(|\eta|\le1/2\), and \eqref{eq:q-feasible}.  Finally,
\(K/n=Sn^{-1/3}\le bq n^{-1/3}\), \(K^2/n^2\le1\), and \(q\ge1\), so
\eqref{eq:remainder-small-explicit} is at most
\(C_{a,b,R_0,\lambda}q\).  Consequently,
\begin{equation}\label{eq:H-small}
\begin{aligned}
  H_n(K,R,q)
  &\le-\frac{K^3}{6n^2}+\frac{\lambda K^2}{2n^{4/3}}-\frac{\lambda^2K}{2n^{2/3}}+C_{\lambda,R_0}\left(\frac Kn+\frac{K^2}{n^2}+|d|n^{-1/3}+1\right)
  \\
  &\le-\frac{S^3}{6}+\frac{|\lambda|S^2}{2}
       +C_{a,b,R_0,\lambda}q.
\end{aligned}
\end{equation}
Performing these cancellations before estimation preserves the cubic term.
Since
\[
  \max_{S\ge0}(|\lambda|S^2/2-S^3/12)=8|\lambda|^3/3,
\]
we have
\(|\lambda|S^2/2\le S^3/12+C_\lambda\).  This constant is absorbed into the
linear term because \(q\ge1\), and \eqref{eq:H-target} follows in Case~1
with \(c_1=1/12\).

\smallskip
\noindent\emph{Case 2: \(K>n/2\).}
Write \(\theta=K/n\in(1/2,1]\).  The estimate
\[
  \log m!=m\log m-m+O\bigl(\log(m+1)\bigr),\qquad m\ge0,
\]
with \(0\log0=0\), is uniform also at \(m=0\).  Applied to \(m=n\) and
\(m=n-K\), it gives
\begin{equation}\label{eq:falling-large}
  \log\frac{(n)_K}{n^K}
  =n\int_0^\theta\log(1-u)\,\dd u+O(\log n),
\end{equation}
where the integral is interpreted continuously at \(\theta=1\).  Indeed,
\[
  \log\frac{(n)_K}{n^K}
  =\log n!-\log(n-K)!-K\log n,
\]
and Stirling's formula reduces the main term to
\[
  -n(1-\theta)\log(1-\theta)-n\theta
  =n\int_0^\theta\log(1-u)\,\dd u.
\]
Also, because \(p_n=(1+\eta)/n\), \(A=O(n^2)\), and
\(\log(1-p_n)=-p_n+O(p_n^2)\),
\begin{equation}\label{eq:large-absent}
\begin{aligned}
  K+A\log(1-p_n)
  &=K-Ap_n+O(Ap_n^2)
  \\
  &=\frac{K^2}{2n}-\eta K+\frac{\eta K^2}{2n}
    +O\!\left(\frac Kn+\frac{|d|}{n}+1\right)
  \\
  &=\frac{n\theta^2}{2}+O(n^{2/3}+q).
\end{aligned}
\end{equation}
Furthermore,
\begin{equation}\label{eq:large-present}
  (K+d)\log(1+\eta)=K\eta+O(K\eta^2+|d|\,|\eta|+|d|\eta^2)=O(n^{2/3}+q).
\end{equation}
Here \(Ap_n^2=O(1)\), and \eqref{eq:q-feasible} implies that all displayed
\(q\)-terms are at most of order \(n^{1/3}\).
Consequently,
\begin{equation}\label{eq:H-large}
  H_n(K,R,q)
  =n g(\theta)+O(n^{2/3}+\log n),
  \qquad
  g(\theta)=\int_0^\theta\log(1-u)\,\dd u+\frac{\theta^2}{2}.
\end{equation}
For \(0<\theta<1\),
\[
  g'(\theta)=\log(1-\theta)+\theta<0,
\]
and
\[
  g(1/2)=\frac12\log2-\frac38<0.
\]
Also \(g(1):=\lim_{\theta\uparrow1}g(\theta)=-1/2\).  Thus \(g\) is bounded
above by a negative constant on \([1/2,1]\).
For large \(n\), \eqref{eq:H-large} gives
\[
  H_n(K,R,q)\le-c_2n\le-c_2S^3,
\]
because \(S^3=K^3/n^2\le n\).  This proves \eqref{eq:H-target} also in
Case~2.  Every constant and the large \(n\) threshold in the two cases are
independent of \(q,K,R\) within \eqref{eq:tuple-range} and \(K\le n\).

On the compact range \eqref{eq:tuple-range}, the product of powers of the
\(x_i\)'s in \eqref{eq:summand-upper} is at most \(C_4^q\).  Combining
\eqref{eq:summand-upper} and \eqref{eq:H-target}, and using \(S\ge aq\),
we obtain, after adjusting constants,
\begin{equation}\label{eq:one-summand-final}
  L_n(\mathbf k,\mathbf r)\le C_5^q n^{-2q/3}e^{-c_3q^3}.
\end{equation}
There are at most \(C_6^q n^{2q/3}\) possible size tuples and at most
\((R_0+1)^q\) surplus tuples.  Summing \eqref{eq:one-summand-final} proves
\eqref{eq:uniform-bound} for \(n\ge n_0\), where \(n_0\) is independent of
\(q\).  It remains only to absorb the finitely many valid indices below
\(n_0\).  If that set is nonempty, put
\[
  n_*:=\max\{n:n_\lambda\le n<n_0\}.
\]
For such \(n\), \(\Xi_n(K_0)\le n_*\).  Keeping \(c\) fixed and replacing
\(C\) by \(\max\{C,n_*e^{cn_*^2}\}\), for \(1\le q\le n_*\) we have
\[
  C^qe^{-cq^3}
  \ge n_*^q e^{cq(n_*^2-q^2)}
  \ge n_*^q
  \ge \E[(\Xi_n(K_0))_q],
\]
while the factorial moment vanishes for \(q>n_*\).  This proves the stated
bound for every \(n\ge n_\lambda\).  If \(|\lambda|\le L\), all Taylor
remainders and error constants above are uniform after replacing
\(|\lambda|\) by \(L\).  Choose one \(n_L\) large enough that all preceding
estimates hold and that \(p_{n,\lambda}\in(0,1/2)\) and
\(|\lambda|n^{-1/3}\le1/2\) for every \(n\ge n_L\) and \(|\lambda|\le L\).
Together with the initial reduction from compact sets to rectangles, this
proves both assertions of the theorem.
\end{proof}

\begin{proposition}[Sharpness of the cubic order]
\label{prop:cubic-optimal}
Fix \(0<a<b<\infty\) and \(R_0\in\mathbb Z_{\ge0}\), and put
\(K_0=[a,b]\times\{0,\ldots,R_0\}\).  There are positive constants
\(A_-,A_+,B_-,B_+\), depending on \(K_0\) and \(\lambda\), such that
\begin{equation}\label{eq:cubic-optimal}
  A_-^q e^{-B_-q^3}
  \le \alpha_\lambda^{(q)}(K_0^q)
  \le A_+^q e^{-B_+q^3},
  \qquad q\ge1.
\end{equation}
Consequently, the power \(q^3\) in the uniform finite \(n\) estimate
\eqref{eq:main-uniform-bound} cannot be replaced by \(q^{3+\varepsilon}\),
for any \(\varepsilon>0\), while retaining a prefactor exponential in \(q\)
and constants independent of \(n,q\).
\end{proposition}

\begin{proof}
The identity
\begin{equation}\label{eq:F-positive-split}
  F(S,\lambda)=\frac{S^3}{24}+\frac{S(S-2\lambda)^2}{8}\ge\frac{S^3}{24}
\end{equation}
and \(S=\sum_i x_i\ge aq\) give the upper bound after bounding the finitely
many functions \(\rho_r\) on \([a,b]\) and integrating over the box.
For the lower bound, restrict the integral to surplus zero and to
\(x_i\in[a,(a+b)/2]\).  On that interval \(\rho_0\) has a positive minimum,
whereas \(S\le bq\) and
\[
  F(S,\lambda)
  \le \frac{S^3}{6}+\frac{|\lambda|S^2}{2}
       +\frac{\lambda^2S}{2}
  \le B_-q^3
\]
after increasing \(B_-\).  Integration gives the lower bound with an
exponential in \(q\) volume factor.  The rectangle \(K_0^q\) is an
\(\alpha_\lambda^{(q)}\)-continuity set because its size coordinate boundary
has Lebesgue measure zero.  If a uniform finite \(n\) upper bound with
\(q^{3+\varepsilon}\) held, taking \(n\to\infty\) at this continuity set
would transfer the same bound to \(\alpha_\lambda^{(q)}(K_0^q)\).
Comparison with the lower bound as \(q\to\infty\) gives a contradiction.
\end{proof}

\begin{corollary}[Finite \(n\) local consequences]
\label{cor:finite-n-consequences}
For every compact \(K\Subset E\), there are \(C_K,c_K>0\) such that, for all
\(n\ge n_\lambda\) and \(m\ge1\),
\begin{equation}\label{eq:overcrowding}
  \Prob\{\Xi_n(K)\ge m\}
  \le \frac{C_K^m}{m!}e^{-c_Km^3}.
\end{equation}
Let \(h:E\to\mathbb C\) be bounded and Borel measurable with
\(\operatorname{supp}h\subset K\), and define
\begin{equation}\label{eq:Gn-def}
  G_n(h)=\E\prod_{z\in\Xi_n}(1+h(z)),
\end{equation}
where atoms are counted with multiplicity.  Then
\begin{equation}\label{eq:Gn-series}
  G_n(h)
  =\sum_{q=0}^\infty\frac1{q!}
    \int_{E^q}\prod_{i=1}^q h(z_i)
    \,\dd\alpha_{\Xi_n}^{(q)}(\mathbf z).
\end{equation}
For every \(T<\infty\), there are \(A=A(K,T)<\infty\) and
\(d=d(K,T)>0\) such that, for every \(Q\in\mathbb Z_{\ge0}\),
\begin{equation}\label{eq:Gn-tail}
\begin{aligned}
  \sup_{\substack{n\ge n_\lambda,\ \operatorname{supp}h\subset K\\
                   \|h\|_\infty\le T}}
  \left|G_n(h)-\sum_{q=0}^Q\frac1{q!}
    \int_{E^q}\prod_{i=1}^q h(z_i)
    \,\dd\alpha_{\Xi_n}^{(q)}(\mathbf z)\right|
  \le A e^{-dQ^3}.
\end{aligned}
\end{equation}
For every Borel set \(B\subset K\), the choice \(h=-\one_B\) gives the
finite \(n\) void probability estimate
\begin{equation}\label{eq:void-tail}
 \sup_{n\ge n_\lambda}
 \left|\Prob\{\Xi_n(B)=0\}
 -\sum_{q=0}^Q\frac{(-1)^q}{q!}
   \E[(\Xi_n(B))_q]\right|
 \le Ae^{-dQ^3},
\end{equation}
where \(A,d\) may be taken from \eqref{eq:Gn-tail} with \(T=1\).
If \(f:E\to\mathbb C\) is bounded and Borel measurable with
\(\operatorname{supp}f\subset K\), then \eqref{eq:Gn-tail} also applies to the
complex local Laplace functional after taking \(h=e^{-f}-1\).  In particular,
for every \(t\ge0\),
\begin{equation}\label{eq:finite-exp-moment}
  \sup_{n\ge n_\lambda}\E e^{t\Xi_n(K)}<\infty.
\end{equation}
\end{corollary}

\begin{proof}
Theorem~\ref{thm:uniform-bound} gives
\(\E[(\Xi_n(K))_q]\le C_K^q e^{-c_Kq^3}\) for all \(q\ge1\) and
\(n\ge n_\lambda\).  On \(\{\Xi_n(K)\ge m\}\), one has
\((\Xi_n(K))_m\ge m!\), so Markov's inequality gives
\eqref{eq:overcrowding}.

Since \(\Xi_n\) has finitely many atoms, the product in \eqref{eq:Gn-def} has
the elementary symmetric expansion.  Taking expectations and using the
definition of the factorial moment measures gives \eqref{eq:Gn-series}.  If
\(\|h\|_\infty\le T\), then the absolute value of the \(q\)th summand in
\eqref{eq:Gn-series} is bounded by
\begin{equation}\label{eq:Gn-term-bound}
  \frac{T^q}{q!}\E[(\Xi_n(K))_q]\le \frac{(C_KT)^q}{q!}e^{-c_Kq^3}.
\end{equation}
This is the same factorial moment estimate that gives \eqref{eq:overcrowding}.
It also gives a summable majorant.  If \(T=0\), the tail is zero.  If \(T>0\),
choose \(0<d_1<c_K\).  Since
\(q\max\{\log(C_KT),0\}\le(c_K-d_1)q^3\) for all sufficiently large \(q\), the
remaining finitely many terms can be absorbed into a constant and
\[
 \sum_{q>Q}\frac{(C_KT)^q}{q!}e^{-c_Kq^3}
 \le C\sum_{q>Q}e^{-d_1q^3}
 \le A e^{-dQ^3}
\]
for any fixed \(0<d<d_1\).  This proves \eqref{eq:Gn-tail}.

Taking \(h=-\one_B\) gives \eqref{eq:void-tail}.  If
\(\operatorname{supp}f\subset K\), then \(e^{-f}-1\) vanishes outside \(K\), so
taking \(h=e^{-f}-1\) gives the asserted complex Laplace functional bound.
Finally, with \(h=(e^t-1)\one_K\), \eqref{eq:Gn-series} and
\eqref{eq:Gn-term-bound} give \eqref{eq:finite-exp-moment}.
\end{proof}

\subsection{Laplace functionals and identification}\label{sec:laplace}

This subsection turns the factorial measure convergence and the all order
bound proved above into convergence of the marked point process.  We prove
tightness, pass to the Laplace functionals, and identify the limit through
the Janson--Spencer Palm description from~\cite{JansonSpencer2007}.

\begin{lemma}[Marked local tightness]\label{lem:marked-tightness}
The sequence \((\Xi_n)_{n\ge n_\lambda}\) is tight in \(\Mp(E)\).
\end{lemma}

\begin{proof}
For every compact \(K\Subset E\), the case \(q=1\) of
Theorem~\ref{thm:uniform-bound} gives
\begin{equation}\label{eq:first-moment-compact}
  \sup_{n\ge n_\lambda}\E\Xi_n(K)<\infty.
\end{equation}
Indeed, for every \(M>0\), Markov's inequality gives
\[
  \sup_{n\ge n_\lambda}\Prob\{\Xi_n(K)>M\}
  \le M^{-1}\sup_{n\ge n_\lambda}\E\Xi_n(K),
\]
and the right hand side tends to zero as \(M\to\infty\).  Since every
relatively compact Borel set is contained in a compact set, the same Markov
bound shows that \((\Xi_n(B))_{n\ge n_\lambda}\) is tight for each such set
\(B\).  The tightness criterion for random measures
\cite[Theorem~4.10]{Kallenberg2017} then yields tightness of
\((\Xi_n)_{n\ge n_\lambda}\) in \(\mathcal M(E)\).
Since \(\Mp(E)\) is closed in \(\mathcal M(E)\), as noted in
Section~\ref{sec:setup}, the sequence is tight in \(\Mp(E)\).
\end{proof}

\begin{proof}[Proof of Theorem~\ref{thm:main}(ii)]
Fix \(f\in C_c^+(E)\), let
\begin{equation}\label{eq:g-def}
  g=1-e^{-f},
  \qquad
  K=\operatorname{supp}f.
\end{equation}
Then \(0\le g\le1\).  For every locally finite point measure \(\xi\),
only finitely many atoms lie in \(K\), and the elementary symmetric function
identity gives the pathwise finite expansion
\begin{equation}\label{eq:laplace-expansion}
  e^{-\langle f,\xi\rangle}=\prod_{z\in\xi}(1-g(z))=\sum_{q=0}^\infty\frac{(-1)^q}{q!}\sum_{z_1,\ldots,z_q\in\xi}^{\neq}\prod_{i=1}^q g(z_i).
\end{equation}
The sum in \eqref{eq:laplace-expansion} is pathwise finite.  In addition,
Corollary~\ref{cor:finite-n-consequences} gives the uniform absolute bound
\[
  \frac{\|g\|_\infty^q}{q!}\E[(\Xi_n(K))_q]
  \le\frac{C_K^q e^{-c_Kq^3}}{q!},
\]
since \(\|g\|_\infty\le1\).  The resulting majorant is summable uniformly in
\(n\), and the expected expansion may be controlled term by term.  If
\(a_{n,q}\) denotes the expected \(q\)th term, then for each fixed \(q\),
Theorem~\ref{thm:main}(i) gives \(a_{n,q}\to a_q\), where
\[
  a_q=\frac{(-1)^q}{q!}\int_{E^q}
  \prod_{i=1}^q(1-e^{-f(z_i)})\,
  \dd\alpha_\lambda^{(q)}(\mathbf z).
\]
The same majorant holds for \(|a_q|\) after passing to the limit.  For
\(Q\ge0\), put
\[
  R(Q)=\sum_{q>Q}\frac{C_K^q e^{-c_Kq^3}}{q!}.
\]
Then \(R(Q)\to0\) and
\[
  \left|\E e^{-\langle f,\Xi_n\rangle}-\sum_{q=0}^Q a_{n,q}\right|\le R(Q),
\]
uniformly in \(n\); the limiting series has the same tail bound.  Since
\(\sum_{q=0}^Q a_{n,q}\to\sum_{q=0}^Q a_q\) for each fixed \(Q\), the
triangle inequality shows that \(\E e^{-\langle f,\Xi_n\rangle}\) converges
to the series on the right hand side of \eqref{eq:enum-Laplace-theorem}.

By Lemma~\ref{lem:marked-tightness}, every subsequence has a further
subsequence converging in law to a locally finite point process.  Along that
further subsequence, the bounded continuous functional
\(\xi\mapsto e^{-\langle f,\xi\rangle}\) converges in expectation.  Formula
\eqref{eq:enum-Laplace-theorem} shows that all subsequential limits have the
same Laplace functional.  The convergence criterion for random measures
\cite[Theorem~4.11]{Kallenberg2017}, applied to constant sequences, shows
that this Laplace functional determines the common law; denote a point
process with this law by \(\Xi_\lambda^{\mathrm{enum}}\).  Since every
subsequence has a further subsequence converging to this law, tightness
implies the full convergence \eqref{eq:enum-convergence}.
\end{proof}

\begin{proof}[Proof of Corollary~\ref{cor:aldous-identification}]
Write
\(\mathcal X_\lambda^{\mathrm A}=\sum_j\delta_{\zeta_j}\) for the unmarked
length projection of \(\Xi_\lambda^{\mathrm A}\).
Theorems~3.1 and~4.1 of Janson--Spencer~\cite{JansonSpencer2007}
give, respectively, the conditional distribution
\(\Prob\{\operatorname{surplus}=r\mid x\}=w_rx^{3r/2}/\Psi(x^{3/2})\) and
the intensity density \(\Lambda_\mu\) of the unmarked excursion length
process at parameter \(\mu\), defined by
\begin{equation}\label{eq:JS-intensity}
  \Lambda_\mu(x)
  =\frac{1}{\sqrt{2\pi}}x^{-5/2}
   \Psi(x^{3/2})e^{-F(x,\mu)}.
\end{equation}
Set \(S_i=x_1+\cdots+x_i\) and \(S_0=0\).  Iterating their Palm formula
\cite[Theorem~8.2]{JansonSpencer2007} gives, for every nonnegative measurable
\(H\),
\begin{equation}\label{eq:JS-Palm-iteration}
  \E\sum_{x_1,\ldots,x_q\in\mathcal X_\lambda^{\mathrm A}}^{\neq}H(x_1,\ldots,x_q)=\int_{(0,\infty)^q}H(\mathbf x)\prod_{i=1}^q\Lambda_{\lambda-S_{i-1}}(x_i)\,\dd\mathbf x.
\end{equation}
Indeed, after the distinguished Palm atom at \(x_1\) is deleted, the
remaining process has parameter \(\lambda-x_1\); induction supplies the
successive shifts.  The intensities are absolutely continuous, so collisions
have zero mass; see also Janson--Spencer
\cite[Corollary~8.7 and Remark~8.8]{JansonSpencer2007} for the corresponding
factorial moment formulas.  Thus the unmarked \(q\)th factorial density is
\begin{equation}\label{eq:JS-length-factorial}
  \prod_{i=1}^q\Lambda_{\lambda-S_{i-1}}(x_i).
\end{equation}
Moreover, that remark gives finite exponential moments for counts in compact
size intervals.  In the Janson--Spencer marked description, conditionally on
the lengths, the marks have the above distribution.  Hence the marked
\(q\)th factorial density at \(((x_i,r_i))_{i=1}^q\) is
\[
  \prod_{i=1}^q\Lambda_{\lambda-S_{i-1}}(x_i)
  \frac{w_{r_i}x_i^{3r_i/2}}{\Psi(x_i^{3/2})}.
\]
The factors \(\Psi(x_i^{3/2})\) cancel.
Using \eqref{eq:F-def} and the change of variables
\(u=S_{i-1}+v\),
\begin{equation}\label{eq:F-telescope}
\begin{aligned}
  F(x_i,\lambda-S_{i-1})
  &=\frac12\int_{S_{i-1}}^{S_i}(u-\lambda)^2\,\dd u,
  \\
  \sum_{i=1}^qF(x_i,\lambda-S_{i-1})
  &=F(S_q,\lambda).
\end{aligned}
\end{equation}
Thus the marked factorial density of \(\Xi_\lambda^{\mathrm A}\) is exactly
\(m_\lambda^{(q)}\) in \eqref{eq:mq-def}; this computation restates the
Janson--Spencer structure in the present notation.

If \(K\Subset E\), the marked count \(\Xi_\lambda^{\mathrm A}(K)\) is
bounded by the unmarked length count in the compact projection of \(K\).
Janson--Spencer's exponential moment bound
\cite[Remark~8.8]{JansonSpencer2007} gives
\(\E 2^{\Xi_\lambda^{\mathrm A}(K)}<\infty\).  Since the absolute value of
the pathwise expansion in \eqref{eq:laplace-expansion} is bounded by
\(2^{\Xi_\lambda^{\mathrm A}(K)}\), expectation may be taken term by term for
\(\xi=\Xi_\lambda^{\mathrm A}\).  Its Laplace functional is therefore the
series on the right hand side of \eqref{eq:enum-Laplace-theorem}.  Hence
Laplace functional uniqueness gives
\(\Xi_\lambda^{\mathrm{enum}}\stackrel d=\Xi_\lambda^{\mathrm A}\), and
\eqref{eq:main-convergence} follows from \eqref{eq:enum-convergence}.
\end{proof}

\section{From local limits to ordered
\texorpdfstring{\(\ell^2\)}{l2} convergence}\label{sec:ordered}

The marked local limit controls compact size windows and finitely many
surplus marks.  This section removes these restrictions.  Summing over all
surplus values gives the unmarked local limit and the small component
estimate, while a macroscopic mass bound controls the upper size tail.
Together these estimates yield ordered \(\ell^2\) convergence and convergence
of the largest marked components.

For \(k\ge1\) and \(r\ge0\), let
\[
  N_n(k,r)
  =\#\{C\in\mathcal C(G(n,p_n)):|C|=k,\ \spn(C)=r\},
  \qquad t_n(k)=\E N_n(k,0).
\]

\begin{lemma}[All surplus component bound]\label{lem:all-surplus}
There is an absolute constant \(A<\infty\) such that, for every \(k\ge1\)
and \(r\ge1\),
\begin{equation}\label{eq:all-surplus-graph-bound}
  c(k,k+r-1)
  \le \left(\frac{A}{r}\right)^{r/2}k^{k+3r/2-2}.
\end{equation}
For every \(0<b<\infty\), there are constants \(C_b,D_b<\infty\) and an
integer \(n_b\), depending also on \(\lambda\), such that, whenever
\(n\ge n_b\) and \(1\le k\le bn^{2/3}\),
\begin{align}
  t_n(k)&\le C_b n k^{-5/2},                                      \label{eq:tree-upper}\\
  \E N_n(k,r)&\le t_n(k)\left(\frac{D_b}{r}\right)^{r/2},
       \qquad r\ge1.                                               \label{eq:surplus-ratio-bound}
\end{align}
\end{lemma}

\begin{proof}
The graph enumeration estimate \eqref{eq:all-surplus-graph-bound} is
Janson--Spencer~\cite[(3.7)]{JansonSpencer2007}; it is uniform in the
surplus.

Janson--Spencer~\cite[(4.1)]{JansonSpencer2007}, applied with surplus zero,
and Cayley's formula give, uniformly for \(1\le k\le bn^{2/3}\),
\[
  t_n(k)=n\frac{k^{k-2}e^{-k}}{k!}e^{-F(kn^{-2/3},\lambda)}
  \bigl(1+O_{b,\lambda}(n^{-1/3})\bigr).
\]
Since \(F\ge0\) and Stirling's formula gives
\(k^{k-2}e^{-k}/k!\le Ck^{-5/2}\) for every \(k\ge1\), this proves
\eqref{eq:tree-upper} for all sufficiently large \(n\).

Applying Proposition~\ref{prop:exact} with \(q=1\) to surplus \(r\) and
dividing by the case \(r=0\), the factors \((n)_k/k!\) cancel.  The absent
edge exponent in \eqref{eq:A-def} decreases by \(r\), while the present edge
exponent increases by \(r\).  Thus
\begin{equation}\label{eq:surplus-tree-ratio}
  \frac{\E N_n(k,r)}{t_n(k)}
  =\frac{c(k,k+r-1)}{k^{k-2}}
   \left(\frac{p_n}{1-p_n}\right)^r.
\end{equation}
For large \(n\), \(p_n/(1-p_n)\le2/n\).  Hence
\eqref{eq:all-surplus-graph-bound} and \(k\le bn^{2/3}\) imply
\[
  \frac{\E N_n(k,r)}{t_n(k)}
  \le\left(\frac{4Ab^3}{r}\right)^{r/2},
\]
which is the estimate behind Janson--Spencer~\cite[(3.8)]{JansonSpencer2007}
written in the present component expectation notation.  This proves
\eqref{eq:surplus-ratio-bound} after renaming the constant.
\end{proof}

Let \(\mathcal X_n\) be the unmarked component size process
\begin{equation}\label{eq:unmarked-process}
  \mathcal X_n=\sum_{C\in\mathcal C(G(n,p_n))}
  \delta_{n^{-2/3}|C|}.
\end{equation}
Recall \(\Psi(t)=\sum_{r\ge0}w_rt^r\).  The unmarked limiting density is
obtained from the marked density \(m_\lambda^{(q)}\) by summing over all
surplus coordinates.  Thus, for \(x_i>0\), define
\begin{equation}\label{eq:unmarked-density}
  \overline m_\lambda^{(q)}(x_1,\ldots,x_q)
  =e^{-F(\sum_i x_i,\lambda)}
   \prod_{i=1}^q\frac{x_i^{-5/2}}{\sqrt{2\pi}}
   \Psi(x_i^{3/2}).
\end{equation}

\begin{proposition}[Unmarked local limit]
\label{prop:unmarked-local}
For every fixed \(q\ge1\), the factorial moment measures of
\(\mathcal X_n\) converge vaguely on \((0,\infty)^q\) to the measure with
density \(\overline m_\lambda^{(q)}\).  Moreover, for every
\(0<a<b<\infty\), there are \(C,c>0\) such that
\begin{equation}\label{eq:unmarked-high-order}
  \E[(\mathcal X_n([a,b]))_q]\le C^q e^{-cq^3},
  \qquad q\ge1,\quad n\ge n_\lambda.
\end{equation}
Consequently,
\begin{equation}\label{eq:unmarked-local-convergence}
  \mathcal X_n\Rightarrow\mathcal X_\lambda^{\mathrm A}
  :=\sum_j\delta_{\zeta_j}
\end{equation}
in the vague topology on locally finite point measures on \((0,\infty)\).
\end{proposition}

\begin{proof}
For fixed component sizes \(\mathbf k=(k_1,\ldots,k_q)\) with
\(\sum_i k_i\le n\) (the remaining tuples contribute zero), division of
\eqref{eq:exact} by its all tree version gives
\begin{equation}\label{eq:tuple-surplus-ratio}
  \frac{L_n(\mathbf k,\mathbf r)}{L_n(\mathbf k,\mathbf0)}
  =\prod_{i=1}^q
   \frac{c(k_i,k_i+r_i-1)}{k_i^{k_i-2}}
   \left(\frac{p_n}{1-p_n}\right)^{r_i}.
\end{equation}
If \(k_i\le bn^{2/3}\), the proof of Lemma~\ref{lem:all-surplus} bounds the
sum of each factor over \(r_i\ge0\) by some \(B_b<\infty\).
Therefore, for all sufficiently large \(n\),
\[
  \E[(\mathcal X_n([a,b]))_q]
  \le B_b^q
      \E[(\Xi_n([a,b]\times\{0\}))_q].
\]
Theorem~\ref{thm:uniform-bound} proves \eqref{eq:unmarked-high-order}.
If there are remaining admissible values, let \(n_*\) be their maximum.
With \(c\) fixed, replacing \(C\) by
\(\max\{C,n_*e^{cn_*^2}\}\) gives the bound for \(q\le n_*\), while the
falling factorial vanishes for \(q>n_*\).

For the vague convergence, fix a continuous test function \(H\) supported in
\([a,b]^q\).  If \(R\) is fixed and each surplus coordinate is restricted to
\(\{0,\ldots,R\}\), Proposition~\ref{prop:local} and the Riemann sum argument
from the proof of Theorem~\ref{thm:main}(i) give the limit of the truncated
factorial integral.  The absolute value of the discarded terms is at most
\(\|H\|_\infty\) times the following expression, which is uniformly small as
\(R\to\infty\): by \eqref{eq:tuple-surplus-ratio},
\[
  \sum_{\mathbf k}\sum_{\max_i r_i>R}L_n(\mathbf k,\mathbf r)\le qB_b^{q-1}\left(\sum_{r>R}(D_b/r)^{r/2}\right)\sum_{\mathbf k}L_n(\mathbf k,\mathbf0),
\]
where the sums over \(\mathbf k\) have \(an^{2/3}\le k_i\le bn^{2/3}\).
The last factor is
\(\E[(\Xi_n([a,b]\times\{0\}))_q]\), hence is bounded uniformly in \(n\) by
Theorem~\ref{thm:uniform-bound}; the tail in \(r\) tends to zero.  Letting
\(k\to\infty\) in \eqref{eq:all-surplus-graph-bound} and using
\eqref{eq:wright} gives \(w_r\le(A/r)^{r/2}\).  Hence, uniformly for
\(x\in[a,b]\),
\[
  \frac{\rho_r(x)}{\rho_0(x)}=w_rx^{3r/2}
  \le\left(\frac{Ab^3}{r}\right)^{r/2},
\]
so the surplus tail of the limiting density also tends to zero uniformly on
\([a,b]^q\).  Letting first \(n\to\infty\) and then \(R\to\infty\) is
therefore justified.  Summing \eqref{eq:local-asymptotic} over all surplus
coordinates gives
\eqref{eq:unmarked-density}, since
\[
  \sum_{r\ge0}\rho_r(x)=\frac{x^{-5/2}}{\sqrt{2\pi}}\Psi(x^{3/2}).
\]
This proves the asserted vague convergence of factorial moment measures.

The tightness and Laplace functional argument of
Subsection~\ref{sec:laplace} applies with \(\mathcal X_n\) and
\eqref{eq:unmarked-high-order} replacing \(\Xi_n\) and
Theorem~\ref{thm:uniform-bound}.  Finally,
\eqref{eq:JS-length-factorial} and \eqref{eq:F-telescope} show that
\(\mathcal X_\lambda^{\mathrm A}\) has precisely the factorial densities
\eqref{eq:unmarked-density}; its local counts have finite exponential
moments by Janson--Spencer~\cite[Remark~8.8]{JansonSpencer2007}.  Thus its
Laplace functional is the resulting absolutely convergent series, proving
\eqref{eq:unmarked-local-convergence}.
\end{proof}

\begin{proposition}[Small components and surplus tails]
\label{prop:tail-estimates}
Recall the coordinates \(X_{n,i}\) of \(\mathbf X_n\) from
\eqref{eq:Xn-def}.
There is \(C=C(\lambda)<\infty\) such that, for \(0<\varepsilon\le1\),
\begin{equation}\label{eq:small-square-mass}
  \sup_{n\ge n_\lambda}
  \E\sum_{i:X_{n,i}\le\varepsilon}X_{n,i}^2
  \le C\sqrt\varepsilon.
\end{equation}
For every \(0<a<b<\infty\),
\begin{equation}\label{eq:surplus-tightness}
  \lim_{R\to\infty}\sup_{n\ge n_\lambda}
  \E\Xi_n([a,b]\times\{R+1,R+2,\ldots\})=0.
\end{equation}
\end{proposition}

\begin{proof}
The series \(1+\sum_{r\ge1}(D_1/r)^{r/2}\) is finite, so
Lemma~\ref{lem:all-surplus}, with \(b=1\), gives
\(\sum_{r\ge0}\E N_n(k,r)\le Cnk^{-5/2}\) for
\(k\le n^{2/3}\) and all sufficiently large \(n\).  Consequently,
\begin{align*}
  \E\sum_{i:X_{n,i}\le\varepsilon}X_{n,i}^2
  &=n^{-4/3}\sum_{k\le\varepsilon n^{2/3}}
    k^2\sum_{r\ge0}\E N_n(k,r)\\
  &\le Cn^{-1/3}\sum_{k\le\varepsilon n^{2/3}}k^{-1/2}
  \le Cn^{-1/3}(\varepsilon n^{2/3})^{1/2}=C\sqrt\varepsilon,
\end{align*}
where the last step uses the integral comparison
\(\sum_{k\le K}k^{-1/2}\le C K^{1/2}\).
For each of the finitely many remaining admissible \(n\), the left side is
zero when \(\varepsilon<n^{-2/3}\); when
\(\varepsilon\ge n^{-2/3}\), its ratio to \(\sqrt\varepsilon\) is bounded
by a finite constant.  Enlarging \(C\) proves
\eqref{eq:small-square-mass} for every admissible \(n\).

For the second assertion, Lemma~\ref{lem:all-surplus} gives, for all large
\(n\),
\[
  \E\Xi_n([a,b]\times\{R+1,R+2,\ldots\})\le\left(\sum_{an^{2/3}\le k\le bn^{2/3}}C_bn k^{-5/2}\right)\sum_{r>R}(D_b/r)^{r/2}.
\]
The first factor is bounded uniformly in \(n\), and the second tends to
zero.  For the finitely many omitted \(n\), the left side vanishes once
\(R\) exceeds the largest feasible surplus.  This proves
\eqref{eq:surplus-tightness}.
\end{proof}

\begin{proposition}[Macroscopic mass]
\label{prop:macro-mass}
There is \(M=M(\lambda)<\infty\) such that
\begin{equation}\label{eq:macro-mass}
  \sup_{n\ge n_\lambda}n^{-2/3}
  \E\sum_{C:\,|C|\ge n^{2/3}}|C|\le M.
\end{equation}
Consequently, for every \(B\ge1\),
\begin{equation}\label{eq:upper-size-tail}
  \sup_{n\ge n_\lambda}\Prob(X_{n,1}>B)\le M/B.
\end{equation}
\end{proposition}

\begin{proof}
Let \(V_n^{\mathrm c}\) be the number of vertices in components of surplus
at least two.  Janson--Spencer~\cite[Lemma~5.2 and
(5.3)]{JansonSpencer2007} prove that
\begin{equation}\label{eq:complex-mass-bound}
  \E V_n^{\mathrm c}=O(n^{2/3});
\end{equation}
their argument uses the all excess enumeration developed in
Janson--Knuth--\L uczak--Pittel~\cite[Lemma~5]{JKLP1993}.

It remains to compare macroscopic tree and unicyclic components with
surplus two components.  The exact single component formula yields, for
\(\ell\in\{0,1\}\), all sufficiently large \(n\), and
\(n^{2/3}\le k\le n\),
\begin{equation}\label{eq:macro-comparison}
  \frac{\E N_n(k,\ell)}{\E N_n(k,2)}
  =\frac{c(k,k+\ell-1)}{c(k,k+1)}
   \left(\frac{1-p_n}{p_n}\right)^{2-\ell}.
\end{equation}
Since \(w_2>0\), \eqref{eq:wright} implies that, for each
\(\ell\in\{0,1\}\), there are \(K,D_\ell<\infty\) such that
\[
  \frac{c(k,k+\ell-1)}{c(k,k+1)}
  \le D_\ell k^{-3(2-\ell)/2},\qquad k\ge K.
\]
Since \((1-p_n)/p_n=O(n)\), for \(k\ge n^{2/3}\) the right hand side of
\eqref{eq:macro-comparison} is bounded by
\[
  D_\ell k^{-3(2-\ell)/2}O(n^{2-\ell})=O(n^{2-\ell}n^{-(2-\ell)})=O(1),
\]
uniformly in \(n\).  After multiplying by \(k\) and summing in \(k\), the
expected macroscopic mass in tree and unicyclic components is therefore at
most a constant times the expected mass in surplus two components, which is
bounded by \(\E V_n^{\mathrm c}\).  Together with
\eqref{eq:complex-mass-bound}, this
proves \eqref{eq:macro-mass} for all sufficiently large \(n\); enlarge the
constant for the finitely many remaining admissible indices.

If \(X_{n,1}>B\), then
\(n^{-2/3}\sum_{|C|\ge n^{2/3}}|C|>B\).  Markov's inequality and
\eqref{eq:macro-mass} give \eqref{eq:upper-size-tail}.
\end{proof}

\begin{proof}[Proof of Corollary~\ref{cor:l2-convergence}(i)]
For \(0<\varepsilon<1\le B<\infty\), let
\(T_{\varepsilon,B}\mathbf x\)
retain the coordinates of \(\mathbf x\) in \([\varepsilon,B]\), arrange
them in nonincreasing order, and pad with zeros.  The limiting point process
has no atom at either endpoint, because its intensity
\eqref{eq:JS-intensity} is absolutely continuous.  At such a point measure,
restriction to \([\varepsilon,B]\), followed by decreasing rearrangement and
zero padding, is continuous as a map into \(\ell^2\).  Hence the continuous
mapping theorem applied to Proposition~\ref{prop:unmarked-local} gives
\begin{equation}\label{eq:truncated-l2-convergence}
  T_{\varepsilon,B}\mathbf X_n
  \Rightarrow T_{\varepsilon,B}\boldsymbol\zeta_\lambda
  \qquad\text{in }\ell^2.
\end{equation}

Campbell's formula and \eqref{eq:JS-intensity} give
\begin{equation}\label{eq:limit-square-integrability}
  \E\sum_i\zeta_i^2=\int_0^\infty x^2\Lambda_\lambda(x)\,\dd x<\infty.
\end{equation}
Indeed, near zero the integrand is \(O(x^{-1/2})\), since
\(\Psi(x^{3/2})\to1\); at infinity, the asymptotic
\(\Psi(t)\sim\tfrac12t^2e^{t^2/24}\)
\cite[(3.4)]{JansonSpencer2007}, together with \eqref{eq:F-def}, gives an
integrable cubic exponential bound.  Thus
\(\boldsymbol\zeta_\lambda\in\ell^2_\downarrow\) almost surely and
\begin{equation}\label{eq:limit-truncation}
  \|\boldsymbol\zeta_\lambda-
  T_{\varepsilon,B}\boldsymbol\zeta_\lambda\|_2
  \longrightarrow0
\end{equation}
almost surely as \(\varepsilon\downarrow0\) and \(B\uparrow\infty\).

For the finite graphs, Propositions~\ref{prop:tail-estimates} and
\ref{prop:macro-mass} give, for every \(\eta>0\),
\begin{align}\label{eq:finite-l2-truncation}
 &\sup_{n\ge n_\lambda}
 \Prob\bigl(\|\mathbf X_n-T_{\varepsilon,B}\mathbf X_n\|_2>\eta\bigr)
 \nonumber\\
 &\qquad\le \frac{M}{B}+\frac{C\sqrt\varepsilon}{\eta^2}.
\end{align}
Indeed, on \(\{X_{n,1}\le B\}\), the squared truncation error is at most
\(\sum_{i:X_{n,i}<\varepsilon}X_{n,i}^2\).  Hence
\[
  \Prob\bigl(\|\mathbf X_n-T_{\varepsilon,B}\mathbf X_n\|_2>\eta\bigr)\le\Prob(X_{n,1}>B)+\Prob\!\left(\sum_{i:X_{n,i}<\varepsilon}X_{n,i}^2>\eta^2\right).
\]
The first term is bounded by \eqref{eq:upper-size-tail}, and Markov's
inequality with \eqref{eq:small-square-mass} bounds the second term.

First let \(n\to\infty\) in \eqref{eq:truncated-l2-convergence}.  Then take
\(\varepsilon_m=m^{-1}\) and \(B_m=m\) for \(m\ge2\), and let
\(m\to\infty\) in
\eqref{eq:finite-l2-truncation} and \eqref{eq:limit-truncation}.  The
converging together theorem~\cite[Chapter~1, Theorem~3.2]{Billingsley1999} proves
\eqref{eq:l2-convergence}.
\end{proof}

\begin{proof}[Proof of Corollary~\ref{cor:l2-convergence}(ii)]
The second factorial moment measure has the absolutely continuous density
\eqref{eq:JS-length-factorial}, so \(\mathcal X_\lambda^{\mathrm A}\) has no
two atoms at the same size almost surely.  Moreover, the mean and variance
asymptotics in Janson--Spencer~\cite[Theorem~1.2]{JansonSpencer2007} imply
that the count \(\mathcal X_\lambda^{\mathrm A}([\varepsilon,\infty))\) has
mean tending to infinity and variance divided by the square of its mean
tending to zero as \(\varepsilon\downarrow0\).  Chebyshev's inequality and
monotonicity in \(\varepsilon\) therefore show that there are almost surely
infinitely many positive atoms.  On the other hand,
\eqref{eq:limit-square-integrability} implies that its largest atom is
finite.  Thus \(\zeta_m>0\) and \(\zeta_1<\infty\) almost surely, so
\(\Prob(\zeta_m\le a)\to0\) as \(a\downarrow0\) and
\(\Prob(\zeta_1\ge b)\to0\) as \(b\uparrow\infty\).  For any
\(\delta>0\), choose continuity points \(0<a<b<\infty\) such that these two
probabilities have sum less than \(\delta\).  Then the largest \(m\)
limiting atoms all lie in \((a,b)\) with probability at least
\(1-\delta\).  By \eqref{eq:l2-convergence}, the largest \(m\) finite \(n\)
atoms then lie in the same interval with probability at least
\(1-\delta-o(1)\).

For fixed \(R\), restrict a point measure to
\([a,b]\times\{0,\ldots,R\}\), order its atoms by decreasing size, retain
the first \(m\), and pad with a fixed cemetery symbol if fewer than \(m\)
atoms remain, using a fixed measurable rule at size ties.  This map is
continuous at every locally finite point measure having no boundary atom and
no two atoms in the window with the same size.
Indeed, its finitely many atoms can then be separated by disjoint
neighborhoods whose boundaries have zero mass; vague convergence preserves
the number and locations of the atoms in these neighborhoods, and the
discrete marks are eventually constant.  The limiting process satisfies
these continuity conditions almost surely.  Thus
\eqref{eq:main-convergence} and the continuous mapping theorem give
convergence of the truncated ordered vectors for every fixed \(a,b,R\).
Furthermore, \eqref{eq:l2-convergence} and the absence of limiting size ties
give
\[
  \Prob\!\left\{\min_{1\le i\le m}(X_{n,i}-X_{n,i+1})=0\right\}
  \longrightarrow0.
\]

The probability that truncation changes one of the first \(m\) finite graph
pairs is at most the probability that one of their sizes lies outside
\((a,b)\), plus the preceding tie probability and
\[
  \Prob\{\Xi_n([a,b]\times\{R+1,R+2,\ldots\})>0\}.
\]
The size probability is at most \(\delta+o(1)\), the tie probability tends
to zero, and the displayed mark probability tends to zero uniformly in
\(n\) as \(R\to\infty\) by \eqref{eq:surplus-tightness} and Markov's
inequality.  For the limit, its
first factorial density gives
\[
  \E\Xi_\lambda^{\mathrm A}([a,b]\times\{R+1,R+2,\ldots\})
  =\sum_{r>R}\int_a^b e^{-F(x,\lambda)}\rho_r(x)\,\dd x\longrightarrow0,
\]
so the corresponding limiting error also tends to zero.  First letting
\(n\to\infty\), then \(R\to\infty\), and finally \(\delta\downarrow0\)
proves \eqref{eq:ranked-marked-convergence}.  The vanishing tie probability
also removes any dependence on the finite \(n\) tie breaking rule.
\end{proof}

\section*{Acknowledgments}
 This work is supported by the National Key R\&D Program of China (No. 2022YFA1006500) and by the National Natural Science Foundation of China (No. 12401171).

\section*{Data Availability Statement}
Data sharing is not applicable to this article, as no datasets were
generated or analysed during the current study.

\section*{Conflict of Interest}
The author declares no conflict of interest.



\begin{thebibliography}{99}

\bibitem{ABG2010}
L.~Addario-Berry, N.~Broutin and C.~Goldschmidt,
Critical random graphs: limiting constructions and distributional
properties,
\emph{Electron. J. Probab.} \textbf{15} (2010), no.~25, 741--775.

\bibitem{ABG2012}
L.~Addario-Berry, N.~Broutin and C.~Goldschmidt,
The continuum limit of critical random graphs,
\emph{Probab. Theory Related Fields} \textbf{152} (2012), 367--406.

\bibitem{Aldous1997}
D.~Aldous,
Brownian excursions, critical random graphs and the multiplicative
coalescent,
\emph{Ann. Probab.} \textbf{25} (1997), 812--854.

\bibitem{Billingsley1999}
P.~Billingsley,
\emph{Convergence of Probability Measures}, 2nd ed.,
Wiley, New York, 1999.

\bibitem{Bollobas1984}
B.~Bollob\'as,
The evolution of random graphs,
\emph{Trans. Amer. Math. Soc.} \textbf{286} (1984), 257--274.

\bibitem{Bollobas2001}
B.~Bollob\'as,
\emph{Random Graphs}, 2nd ed.,
Cambridge University Press, Cambridge, 2001.

\bibitem{BJR2007}
B.~Bollob\'as, S.~Janson and O.~Riordan,
The phase transition in inhomogeneous random graphs,
\emph{Random Structures Algorithms} \textbf{31} (2007), 3--122.

\bibitem{DVJ2008}
D.~J. Daley and D.~Vere-Jones,
\emph{An Introduction to the Theory of Point Processes. Vol.~II:
General Theory and Structure}, 2nd ed.,
Springer, New York, 2008.

\bibitem{DeAmbroggio2022}
U.~De Ambroggio,
An elementary approach to component sizes in critical random graphs,
\emph{J. Appl. Probab.} \textbf{59} (2022), no.~4, 1228--1242.

\bibitem{DeAmbroggio2024}
U.~De Ambroggio,
A simple path to component sizes in critical random graphs,
\emph{SIAM J. Discrete Math.} \textbf{38} (2024), no.~2, 1492--1525.

\bibitem{DeAmbroggioRoberts2022}
U.~De Ambroggio and M.~I. Roberts,
Unusually large components in near-critical Erd\H{o}s--R\'enyi graphs via
ballot theorems,
\emph{Combin. Probab. Comput.} \textbf{31} (2022), no.~5, 840--869.

\bibitem{ER1960}
P.~Erd\H{o}s and A.~R\'enyi,
On the evolution of random graphs,
\emph{Publ. Math. Inst. Hungar. Acad. Sci.} \textbf{5} (1960), 17--61.

\bibitem{Janson2007}
S.~Janson,
Brownian excursion area, Wright's constants in graph enumeration, and
other Brownian areas,
\emph{Probab. Surv.} \textbf{4} (2007), 80--145.

\bibitem{JKLP1993}
S.~Janson, D.~E. Knuth, T.~\L uczak and B.~Pittel,
The birth of the giant component,
\emph{Random Structures Algorithms} \textbf{4} (1993), 233--358.

\bibitem{JLR2000}
S.~Janson, T.~\L uczak and A.~Ruci\'nski,
\emph{Random Graphs},
Wiley-Interscience, New York, 2000.

\bibitem{JansonSpencer2007}
S.~Janson and J.~Spencer,
A point process describing the component sizes in the critical window of
the random graph evolution,
\emph{Combin. Probab. Comput.} \textbf{16} (2007), 631--658.

\bibitem{Kallenberg2017}
O.~Kallenberg,
\emph{Random Measures, Theory and Applications},
Springer, Cham, 2017.

\bibitem{Luczak1990}
T.~\L uczak,
Component behavior near the critical point of the random graph process,
\emph{Random Structures Algorithms} \textbf{1} (1990), 287--310.

\bibitem{LPW1994}
T.~\L uczak, B.~Pittel and J.~C. Wierman,
The structure of a random graph at the point of the phase transition,
\emph{Trans. Amer. Math. Soc.} \textbf{341} (1994), 721--748.

\bibitem{NachmiasPeres2010}
A.~Nachmias and Y.~Peres,
The critical random graph, with martingales,
\emph{Israel J. Math.} \textbf{176} (2010), 29--41.

\bibitem{Roberts2018}
M.~I. Roberts,
The probability of unusually large components in the near-critical
Erd\H{o}s--R\'enyi graph,
\emph{Adv. Appl. Probab.} \textbf{50} (2018), no.~1, 245--271.

\bibitem{Spencer1997}
J.~Spencer,
Enumerating graphs and Brownian motion,
\emph{Comm. Pure Appl. Math.} \textbf{50} (1997), 291--294.

\bibitem{vdH2017}
R.~van der Hofstad,
\emph{Random Graphs and Complex Networks. Vol.~1},
Cambridge University Press, Cambridge, 2017.

\bibitem{HKM2009}
R.~van der Hofstad, W.~Kager and T.~M\"uller,
A local limit theorem for the critical random graph,
\emph{Electron. Commun. Probab.} \textbf{14} (2009), 122--131.

\bibitem{Wright1977}
E.~M. Wright,
The number of connected sparsely edged graphs,
\emph{J. Graph Theory} \textbf{1} (1977), 317--330.

\bibitem{Wright1980}
E.~M. Wright,
The number of connected sparsely edged graphs. III. Asymptotic results,
\emph{J. Graph Theory} \textbf{4} (1980), 393--407.

\end{thebibliography}
\end{document}